\documentclass{amsart}
\usepackage{graphicx}
\newtheorem{thm}{Theorem}[section]

\newtheorem{lem}[thm]{Lemma}
\newtheorem{prop}[thm]{Proposition}
\newtheorem{conj}[thm]{Conjecture}
\newtheorem{obs}[thm]{Observation}
\theoremstyle{definition}
\newtheorem{defn}[thm]{Definition}
\theoremstyle{remark}

\theoremstyle{definition}
\newtheorem{notn}[thm]{Notation}

\numberwithin{equation}{section}

\newcommand{\Text}[1]{\text{\textnormal{#1}}}
\begin{document}

\title{On Serial Symmetric Exchanges of Matroid Bases}%
\author{Daniel Kotlar and Ran Ziv}%
\address{Computer Science Department, Tel-Hai College, Upper Galilee 12210, Israel}%
\email{dannykot@telhai.ac.il, ranziv@telhai.ac.il}%

\thanks{The second author thanks Ron Aharoni and Eli Berger for fruitful discussions concerning the problem.}%

\begin{abstract}
We study some properties of a serial (i.e. one-by-one) symmetric exchange of elements of two disjoint bases of a matroid. We show that any two elements of one base have a serial symmetric exchange with some two elements of the other base. As a result, we obtain that any two disjoint bases in a matroid of rank 4 have a full serial symmetric exchange.
\end{abstract}
\maketitle
\section{Introduction}

A matroid is a hereditary family $M$ of subsets (called independent) of a finite ground set $S$ that satisfies an exchange axiom: If $A,B \in M$ and $|B|>|A|$, then there exists $x\in B\setminus A$ such that $A \cup{x}\in M$. A maximal independent set is called a \emph{base}. An element $x\in S$ is spanned by $A$ if either $x\in A$ or $I\cup \{x\}\not\in M$ for some independent set $I\subseteq A$. The rank of $A\subseteq S$, denoted here by $\rho(A)$, is the size of a maximal independent subset in $A$. We also adopt the common notation $A+x$ for $A\cup\{x\}$ and $A-x$ for $A\setminus\{x\}$. A \emph{circuit} is a minimal dependent set. When $I$ is independent but $I+x$ is not, we shall denote the unique minimal subset of $I$ that spans $x$ (called the \emph{support} of $x$) by $C(I,x)$.  We denote by $C^{+}(I,x)$ the circuit $C(I,x)+x$. For further knowledge and details about matroid theory the reader is referred to Oxley~\cite{Oxley11} and Welsh~\cite{Welsh76}.

The main goal of this paper is to examine the following conjecture:
\begin{conj}\label{11::11}
Let $B_{1}$ and $B_{2}$ be two disjoint bases of a matroid $M$ of rank $n$. There exists an ordering $\{b_{1}\prec b_{2}\prec\cdots\prec b_{2n}\}$ of the elements of $B_{1}\cup B_{2}$, such that the first $n$ elements belong to $B_{1}$ and, for every $i=1,2,...,2n$, the set $\{b_{i\bmod{2n}},b_{(i+1)\bmod{2n}},\ldots,b_{(i+n-1)\bmod{2n}}\}$ is a base.
\end{conj}

As far as we know, this intriguing conjecture was posed implicitly in an early paper by Gabow~\cite{Gabow76}, and later in the more general form of ``cyclic base orders'' by Kajitani and Sugishta~\cite{Kaj83} and Weidemann~\cite{Wei}.
Partial results and some further possible generalizations appear in the works of Kajitaniet al. \cite{Kaj88}, Cordovil and Moreira \cite{Cor93} and van den Heuvel and Thomass\'{e} \cite{HT}. Recently Bonin~\cite{Bonin10} proved it for sparse paving matroids.

It is easy to see that Conjecture~\ref{11::11} may be reformulated in the following equivalent form:
\begin{conj}\label{11::22}
Let $A$ and $B$ be two disjoint bases of a matroid $M$ of rank $n$. There exists an ordering $\{a_{1}\prec a_{2}\prec\cdots\prec a_{n}\}$ of the elements of $A$ and an ordering $\{b_{1}\prec b_{2}\prec\cdots\prec b_{n}\}$ of the elements of $B$, such that for every $i=1,2,\ldots,n$ both $(A\setminus\{a_{1},\ldots,a_{i}\})\cup\{b_{1},\ldots,b_{i}\}$ and $(B\setminus\{b_{1},\ldots,b_{i}\})\cup\{a_{1},\ldots,a_{i}\}$ are bases of $M$.
\end{conj}
In \cite{Cor93} the same problem was cast in terms of the ``base-cobase graph'':
\begin{defn}
A matroid $M$ whose ground set $S$ is a disjoint union of two bases is called a \emph{block matroid}. The \emph{base-cobase graph} $G(V,E)$ of a block matroid $M$ consists of a set of vertices $V=\{B\in M|B \,\textrm{and}\,S\setminus B\,\textrm{are bases}\}$, where the unordered pair $(B,B')$ is an edge if and only if $B$ and $B'$ are bases in $V$ differing by exactly two elements, i.e. $|B\bigtriangleup B'|=2$.
\end{defn}
Under these terms, Conjecture~\ref{11::22}, restricted to block matroids, takes the following form \cite{Cor93}:
\begin{conj}\label{11::33}
If $G$ is the base-cobase graph of a block matroid of rank $n$, then the diameter of $G$ is equal to $n$.
\end{conj}
Conjecture~\ref{11::33} was proved for graphic block matroids by Farber, Richter and Shank \cite{FRS85} (with a modification by Weidemann \cite{Wei}), and independently by  Kajitani et al. \cite{Kaj88} and Cordovil and Moreira \cite{Cor93}. It was also proved for transversal block matroids by Farber \cite{Far89}. As far as we are aware, it is still unknown whether the base-cobase graph is connected for all block matroids.

\section{Some Definitions and Lemmas}
We begin with a few definitions and notations, some of which were introduced by Gabow \cite{Gabow76}. For convenience, we slightly modify the terminology used there.
Let $M$ be a matroid of rank $n$, and let $A=\{a_{1},a_{2},\ldots,a_{n}\}$ and $B=\{b_{1},b_{2},\ldots,b_{n}\}$ be two disjoint bases in $M$.
\begin{defn}
Let $X\subseteq A$ and $Y\subseteq B$ such that $|X|=|Y|=k$.

(i) The pair $(X,Y)$ is a \emph{serial exchange relative to the base $A$} if there exist orderings $X=\{a_{1}\prec a_{2}\prec\cdots \prec a_{k}\}$ and $Y=\{b_{1}\prec b_{2}\prec\cdots\prec b_{k}\}$ so that for each $i=1,2,\ldots,k$,  $(A\setminus\{a_{1},\ldots,a_{i}\})\cup\{b_{1},\ldots,b_{i}\}$ is a base.

(ii) The pair $(X,Y)$ is a \emph{serial symmetric exchange relative to the bases $A$ and $B$} if there exist orderings $X=\{a_{1}\prec a_{2}\prec\cdots\prec a_{k}\}$ and $Y=\{b_{1}\prec b_{2}\prec\cdots\prec b_{k}\}$ so that for each $i=1,2,\ldots,k$, both $(A\setminus\{a_{1},\ldots,a_{i}\})\cup\{b_{1},\ldots,b_{i}\}$ and $(B\setminus\{b_{1},\ldots,b_{i}\})\cup\{a_{1},\ldots,a_{i}\}$ are bases.
\end{defn}
When $X=\{a\}$ and $Y=\{b\}$ we call the pair $(a,b)$ a \emph{symmetric exchange relative to $A$ and $B$} and say that $a$ and $b$ are \emph{symmetrically exchangeable}.

Note that Conjecture~\ref{11::22} states that any pair of bases ($A$,$B$) of a matroid $M$ is a serial symmetric exchange (relative to themselves).

In terms of \cite{Gabow76} a serial symmetric exchange is a sequence of one-element sets $\{a_{i}\}$ and $\{b_{i}\}$, $i=1,\ldots,k$, that constitutes a serial $A$-exchange and a serial $B$-exchange simultaneously.

We list two well-known properties of symmetric exchanges:
\begin{obs}\label{obs4}
$a\in A$ and $b\in B$ are symmetrically exchangeable relative to $A$ and $B$ if and only if $a\in C(A,b)$ and $b\in C(B,a)$.
\end{obs}
\begin{obs}\label{Obs2}
Given two bases $A$ and $B$ of a matroid $M$, for each $a\in A$ there exists $b\in B$, so that $a$ is symmetrically exchangeable with $b$ relative to $A$ and $B$.
\end{obs}

\begin{defn}
We call two symmetric exchanges $(a_{1},b_{1})$ and $(a_{2},b_{2})$ (both relative to the same two bases) \emph{disjoint} if $a_{1}\neq a_{2}$ and $b_{1}\neq b_{2}$.
\end{defn}

Assume that all elements of $A$ are symmetrically exchangeable with the same element $b\in B$. Let $b'\neq b$ be another element of $B$. By Observation~\ref{Obs2} there exists $a'\in A$ so that $(b',a')$ is a symmetric exchange relative to $B$ and $A$. Let $a\in A$ be such that $a\neq a'$. Then $(a,b)$ and $(a',b')$ are two disjoint symmetric exchanges. We conclude:
\begin{prop}\label{obs3}
If $A$ and $B$ are two disjoint bases of a matroid $M$ with $\rho(M)>1$, then there always exist at least two disjoint symmetric exchanges relative to $A$ and $B$.
\end{prop}

Now suppose $\rho(M)=3$. Let $A=\{a_{1},a_{2},a_{3}\}$ and $B=\{b_{1},b_{2},b_{3}\}$. By Proposition~\ref{obs3} we may assume that $(a_{1},b_{1})$ and $(a_{3},b_{3})$ are disjoint symmetric exchanges relative to $A$ and $B$. Exchanging the pair $(a_{1},b_{1})$ relative to $A$ and $B$ we obtain the two bases $A'=\{b_{1},a_{2},a_{3}\}$ and $B'=\{a_{1},b_{2},b_{3}\}$. Exchanging the pair $(a_{3},b_{3})$ relative to $A$ and $B$ we obtain the two bases $A''=\{b_{1},b_{2},a_{3}\}$ and $B''=\{a_{1},a_{2},b_{3}\}$. $A''$ and $B''$ can also be obtained by performing the exchange $(a_{2},b_{2})$ relative to $A'$ and $B'$. Thus, as is already known (see \cite{Gabow76}), Conjecture~\ref{11::22} holds in the case where $\rho(M)=3$.

\vspace{5 mm}
The following lemma, also known as the circuit elimination axiom, (see \cite{Welsh76} or \cite{Oxley11}) will be used frequently here for proving exchange properties of bases:
\begin{lem}\label{lem1}
If $C_{1}$ and $C_{2}$ are two circuits so that $x\in C_{1}\cap C_{2}$ and $y\in C_{1}\setminus C_{2}$, then there exists a circuit $C_{3}\subset C_{1}\cup C_{2}$ such that $x\not\in C_{3}$ and $y\in C_{3}$.
\end{lem}

We look at the directed bipartite graph whose parts are $A$ and $B$ and there is an edge from $a \in A$ to $b\in B$ if and only if $b\in C(B,a)$ and an edge from $b\in B$ to $a\in A$ if and only if $a\in C(A,b)$. For $a', a''\in A$, we look at the directed paths of length two from $a'$ to $a''$ and consider the middle elements of these paths (the elements of $B$ that "connect" $a'$ to $a''$), for which we introduce the following notation:
\begin{notn}\label{notn1}
Let $A$ and $B$ be two disjoint bases of a matroid $M$ and let $a',a''\in A$ be two distinct elements. Let
\begin{equation*}
Conn(a',a'',A,B)=\{b\in B|b\in C(B,a')\;\textrm{and}\;a''\in C(A,b)\}
\end{equation*}
\end{notn}

\begin{prop}\label{lem2}
For any $a', a''\in A$, $|Conn(a',a'',A,B)|\neq1$
\end{prop}
\begin{proof}
By restricting $M$ to $A\cup B$ we may assume that $M$ is a block matroid. Thus $Conn(a',a'',A,B)$ is the intersection of the circuit $C^+(B,a')$ and the cocircuit $\{b\in B| a''\in C(b,A)\}\cup\{a''\}$. The result follows from the fact that the intersection of a circuit and a cocircuit is never a singleton.
\end{proof}
When a serial exchange relative to the base $B$ is carried out and some $a_{i}$s replace $b_{i}$s in $B$, one by one, it is natural to ask how the ``serial'' supports $C(B-b_{1}+a_{1}-\ldots-b_{i-1}+a_{i-1},a_{i})$ are related to the ``original'' supports $C(B,a_{i})$.
The following lemma, which may have its own interest, describes such a relation:
\begin{lem}\label{lem2:5}
Let $(X=\{a_{1}\prec \cdots\prec a_{m}\}$,$Y=\{b_{1}\prec\cdots\prec b_{m}\})$ be a serial exchange relative to $B$. Let $A_{0}=B_{0}=\emptyset$ and for $k=1,\dots,m$ let $A_{k}=\{a_{1}\prec\cdots\prec a_{k}\}$ and $B_{k}=\{b_{1}\prec\cdots\prec b_{k}\}$. Then
\begin{equation}\label{eq2:1}
\bigcup_{i=1}^{k}C(B,a_{i})=\bigcup_{i=1}^{k}C((B\setminus B_{i-1})\cup A_{i-1},a_{i})\cap B\hspace{10 mm}(k=1,\ldots,m).
\end{equation}
\begin{proof}
We prove, by induction on $k$, that the set on left of (\ref{eq2:1}) is contained in the set on the right. Let $b\in \bigcup_{i=1}^{k}C(B,a_{i})$. If $b\in\bigcup_{i=1}^{k-1}C(B,a_{i})$, then $b\in \bigcup_{i=1}^{k-1}C((B\setminus B_{i-1})\cup A_{i-1},a_{i})$ by the induction hypothesis. Hence we may assume that $b\not\in C(B,a_{i})$ for all $i<k$ and $b\in C(B,a_{k})$. If $b\in C((B\setminus B_{k-1})\cup A_{k-1},a_{k})$, then we are done, so we assume $b\in C(B,a_{k})\setminus C((B\setminus B_{k-1})\cup A_{k-1}, a_{k})$. We now construct a recursively defined sequence of circuits $D_i$, $i=k-1,k-2,\ldots$. By Lemma~\ref{lem1}, there is a circuit $D_{k-1}$ in $C(B,a_{k})\cup C((B\setminus B_{k-1})\cup A_{k-1}, a_{k})$ such that $b\in D_{k-1}$, but $a_{k}\not\in D_{k-1}$. Hence $D_{k-1}\subseteq A_{k-1}\cup B$. If $b\in C((B\setminus B_{k-2})\cup A_{k-2}, a_{k-1})$, then we are done. So we assume that $b\in D_{k-1}\setminus C((B\setminus B_{k-2})\cup A_{k-2}, a_{k-1})$. If $a_{k-1}\not\in D_{k-1}$ let $D_{k-2}=D_{k-1}$. Otherwise, we apply Lemma~\ref{lem1} on the circuits $D_{k-1}$ and $C^{+}((B\setminus B_{k-2})\cup A_{k-2}, a_{k-1})$ to obtain a circuit $D_{k-2}$ containing $b$ and excluding $a_{k-1}$. We proceed in the same manner, obtaining a sequence of circuits $D_{k-i}\subset A_{k-i}$ for $i=1,2,\ldots$. The process must terminate by finding some $j<k$ such that $b\in C((B\setminus B_{j-1})\cup A_{j-1}, a_{j})$. Otherwise we reach a contradiction by obtaining a circuit $D_{j}$ containing only one element $a_{i}$ with $i<k$, contradicting the assumption that $b\not\in C(B,a_{i})$ for all $i<k$.
The opposite containment is proved in a similar manner and is left for the reader.
\end{proof}
\end{lem}
The following lemma relates spanning sets before and after performing an exchange. It is an extension of Lemma 2 in \cite{Gabow76}:
\begin{lem}\label{lem3}
Let $a_{1},a_{2}\in A$ and $b_{1},b_{2}\in B$. Suppose $B'=B-b_{1}+a_{1}$ is a base and either $b_{2}\not\in C(B,a_{1})$ or $b_{1}\not\in C(B,a_{2})$. Then $b_{2}\in C(B',a_{2})$ if and only if $b_{2}\in C(B,a_{2})$.
\end{lem}
\begin{proof}
We show that if $b_{2}\in C(B,a_{2})$ and $b_{2}\not\in C(B',a_{2})$, then $b_{1}\in C(B,a_{2})$ and $b_{2}\in C(B,a_{1})$. Clearly, $b_{1}\in C(B,a_{2})$, otherwise $C(B',a_{2})=C(B,a_{2})$, contrary to our assumption. Since $b_{2}\in C(B,a_{2})$, $b_{2}$ is contained in the left hand side of (\ref{eq2:1}) (with $k=2$). When $k=2$ the right hand side of (\ref{eq2:1}) consists of two terms: $C(B,a_{1})$ and $C(B-b_{1}+a_{1},a_{2})\cap B$. Since we assumed that $b_{2}$ is not contained in the second term, it must be in the first term, namely $b_{2}\in C(B,a_{1})$. The other direction, where we assume that $b_{2}\not\in C(B,a_{2})$ and $b_{2}\in C(B',a_{2})$, is handled similarly and is left to the reader.
\end{proof}

The last lemma in this section states that after performing a symmetric exchange the inserted element inherits its support  and the set of elements it supports from the original one.
\begin{lem}\label{lem4}
Suppose that $a\in A$ and $b\in B$ are symmetrically exchangeable relative to $A$ and $B$, and let $b'\in B-b$. Then
\begin{enumerate}
\item[\Text{(i)}]  $C(B-b+a,b)=C(B,a)-b+a$; and
\item[\Text{(ii)}]  $b\in C(A-a+b,b')$ if and only if $a\in C(A,b')$
\end{enumerate}
\end{lem}
\begin{proof}
(i) Note that $C(B,a)+a=C(B-b+a,b)+b$. Subtracting $b$ we get (i).

(ii) Suppose that $a\not\in C(A,b')$. Then the support of $b'$ in A remains unchanged after replacing $a$ with $b$. Thus $b\not\in C(A-a+b,b')$. To show the other direction, suppose $b\not\in C(A-a+b,b')$. We go back and replace $b$ with $a$. Again, the support of $b'$ remains unchanged so $a\not\in C(A,b')$.
\end{proof}

\section{Serial Symmetric Exchanges}
A well-known and basic result on symmetric exchanges between subsets of two bases is the following lemma (\cite{Bry}, \cite{Greene}, \cite{Wood}):
\begin{lem}\label{BElem}
Let $A$ and $B$ be two bases of a matroid $M$. For any $A_{1}\subset A$ there exists $B_{1}\subset B$ such that $(A\setminus A_{1})\cup B_{1}$ and $(B\setminus B_{1})\cup A_{1}$ are bases.
\end{lem}
We conjecture that the subsets $A_{1}$ and $B_{1}$ in Lemma~\ref{BElem} can be exchanged serially:
\begin{conj}\label{conj32}
Let $A$ and $B$ be two bases of a matroid $M$. For any $A_{1}\subset A$ there exists $B_{1}\subset B$ such that $A_{1}$ and $B_{1}$ form a serial symmetric exchange relative to $A$ and $B$.
\end{conj}

the main result of this paper shows that Conjecture~\ref{conj32} holds for subsets of size two:
\begin{thm}\label{thm1}
Let $A$ and $B$ be two bases of a matroid $M$. For any $A_{1}\subset A$ of size two there exists $B_{1}\subset B$ such that $A_{1}$ and $B_{1}$ form a serial symmetric exchange relative to $A$ and $B$.
\end{thm}
\begin{proof}
Let $\{a_{1},a_{2}\}\subset A$ and suppose $a_{1}$ is symmetrically exchangeable, relative to $A$ and $B$, with $b_{1}\in B$. Let $A'=A-a_{1}+b_{1}$ and $B'=B-b_{1}+a_{1}$. Assume that $a_{2}$ is not symmetrically exchangeable, relative to $A'$ and $B'$, with any of $\{b_{2}, b_{3},\ldots\}$ (otherwise we are done). Hence, by Observation~\ref{Obs2}, $a_{2}$ must be symmetrically exchangeable with $a_{1}$ (relative to $A'$ and $B'$). Thus $a_{2}\in C(A',a_{1})$ and $a_{1}\in C(B',a_{2})$. Since $a_{1}\in C(B',b_{1})$ and $b_{1}\in C(A',a_{1})$ ($b_{1}$ and $a_{1}$ are symmetrically exchangeable relative to $A'$ and $B'$), $a_{1}\in Conn(b_{1},a_{2},A',B')$ and $a_{1}\in Conn(a_{2},b_{1},A',B')$ (Notation~\ref{notn1}). By Proposition~\ref{lem2} there must be some $b_{i}$ with $i\neq1$, say $b_{2}$, such that
\begin{equation}\label{eq1}
b_{2}\in C(B',b_{1})\hspace{10 mm}\textrm{and}\hspace{10 mm}a_{2}\in C(A',b_{2}),
\end{equation}
and there is some $b_{j}$ with $j\geq3$, say $b_{3}$, such that
\begin{equation}\label{eq2}
b_{3}\in C(B',a_{2})\hspace{10 mm}\textrm{and}\hspace{10 mm}b_{1}\in C(A',b_{3})
\end{equation}
($j\neq2$ since we assume that $a_{2}$ is not symmetrically exchangeable with $b_{2}$ relative to $A'$ and $B'$).

We will show that $\{a_{1},a_{2}\}$ and $\{b_{2},b_{3}\}$ form a serial symmetric exchange as desired.

Since we assumed that $a_{2}$ is not symmetrically exchangeable with either $b_{2}$ or $b_{3}$ relative to $A'$ and $B'$, it follows from (\ref{eq1}), (\ref{eq2}) and Observation~\ref{obs4} that
\begin{equation}\label{eq3}
b_{2}\not\in C(B',a_{2})\hspace{10 mm}\textrm{and}\hspace{10 mm}a_{2}\not\in C(A',b_{3}),
\end{equation}
The remainder of the proof is as follows. We distinguish two separate cases. If $b_{1}$ and $b_{2}$ are symmetrically exchangeable relative to $A'$ and $B'$ we exchange them and show that after this exchange $a_{2}$ becomes symmetrically exchangeable with $b_{3}$ and we are done. If $b_{1}$ and $b_{2}$ are not symmetrically exchangeable relative to $A'$ and $B'$, we exchange $a_{2}$ and $a_{1}$, relative to $A'$ and $B'$, and show that after this exchange $b_{1}$ becomes symmetrically exchangeable with $b_{2}$. We exchange them and show that we can proceed from here as in case 1.

\underline{Case 1}: $b_{1}$ and $b_{2}$ are symmetrically exchangeable relative to $A'$ and $B'$. Thus, by Observation~\ref{obs4},
\begin{equation}\label{eq4}
b_{1}\in C(A',b_{2}).
\end{equation}
We look at the circuits $C^{+}(A',b_{2})$ and $C^{+}(A',b_{3})$. From (\ref{eq1}), (\ref{eq2}), (\ref{eq3}), (\ref{eq4}) and Lemma~\ref{lem1} it follows that there is a circuit $D$ containing $a_{2}$ and excluding $b_{1}$. $D$ must contain at least one of $b_{2}$ and $b_{3}$, otherwise it would contain only elements of $A$. We show that it must contain both. If $b_{2}\not\in D$, then $D$ consists of $b_{3}$, $a_{2}$ and possibly some other $a_{i}$s with $i>2$. It follows that $D-b_{3}=C(A', b_{3})$. In particular $a_{2}\in C(A', b_{3})$, contrary to (\ref{eq3}). If $b_{3}\not\in D$, then $D$ consists of $b_{2}$, $a_{2}$ and possibly some other $a_{i}$s with $i>2$. It follows that $D-b_{2}=C(A', b_{2})$. In particular $b_{1}\not\in C(A', b_{2})$, contrary to (\ref{eq4}).

We now perform the exchange between $b_{1}$ and $b_{2}$, relative to $A'$ and $B'$, and obtain the bases $A^{*}=A'-b_{1}+b_{2}$ and $B^{*}=B'-b_{2}+b_{1}$. Note that $A^{*}$ and $B^{*}$ can be obtained from $A$ and $B$ by the single exchange $(a_{1}, b_{2})$. Now, since the circuit $D$ from the previous paragraph consists of $b_{2}$, $b_{3}$, $a_{2}$ and possibly some other $a_{i}$s with $i>2$, we have that $D-b_{3}=C(A^{*},b_{3})$. In particular
\begin{equation}\label{eq5}
a_{2}\in C(A^{*},b_{3}).
\end{equation}
Since $b_{2}\not\in C(B',a_{2})$ (by (\ref{eq3})) we have that $C(B',a_{2})=C(B^{*},a_{2})$ and from (\ref{eq2}) we obtain
\begin{equation}\label{eq6}
b_{3}\in C(B^{*},a_{2}).
\end{equation}
Equations~(\ref{eq5}) and (\ref{eq6}) imply that $a_{2}$ and $b_{3}$ are symmetrically exchangeable relative to $A^{*}$ and $B^{*}$. Since $A^{*}$ and $B^{*}$ can be obtained from $A$ and $B$ by the single symmetric exchange $(a_{1},b_{2})$, we conclude that the sets $\{a_{1},a_{2}\}$ and $\{b_{2},b_{3}\}$ constitute a serial symmetric exchange relative to $A$ and $B$.

\underline{Case 2}: $b_{1}$ and $b_{2}$ are not symmetrically exchangeable relative to $A'$ and $B'$. From (\ref{eq1}) and Observation~\ref{obs4} we must have that
\begin{equation}\label{eq7}
b_{1}\not\in C(A',b_{2}).
\end{equation}
We recall that $C^{+}(A',a_{1})$ consists, besides $a_1$, of $b_{1},\,a_{2}$ and possibly other $a_{i}$s with $i>2$. Also, $C^{+}(A',b_{2})$ consists of $a_{2}$ and possibly other $a_{i}$s with $i>2$, but excludes $b_{1}$, by (\ref{eq7}). Applying Lemma~\ref{lem1} to these two circuits we obtain a circuit $D'$ containing $a_{1}$, possibly other $a_{i}$s with $i>2$, at least one of $b_{1}$ and $b_{2}$, and excluding $a_{2}$. We claim that $D'$ must contain both $b_{1}$ and $b_{2}$. First we observe that since $C(A',b_{2})$ contains $a_{2}$ and excludes $b_{1}$, $a_{2}\in C(A',b_{2})=C(A,b_{2})$. Suppose $b_{1}\not\in D'$. Then $D'$ consists of $b_{2}$, $a_{1}$ and possibly some other $a_{i}$s with $i>2$. Hence $D'-b_{2}=C(A,b_{2})$, which means that $a_{2}\not\in C(A,b_{2})=C(A',b_{2})$, contradicting (\ref{eq1}). Now suppose $b_{2}\not\in D'$. Then $D'$ consists of $b_{1}$, $a_{1}$ and possibly some other $a_{i}$s with $i>2$. Thus $D'-a_{1}=C(A',a_{1})$ and hence $a_{2}\not\in C(A',a_{1})$, contradicting the assumption that $a_{2}$ and $a_{1}$ are symmetrically exchangeable relative to $A'$ and $B'$. Hence $b_{1}, b_{2}\in D'$.

We now exchange $a_{2}$ and $a_{1}$ relative to $A'$ and $B'$ and obtain the bases $A''=A'-a_{2}+a_{1}$ and $B''=B'-a_{1}+a_{2}$. Note that $A''$ and $B''$ can be obtained from $A$ and $B$ by the single exchange $(a_{2},b_{1})$. Since $a_{1},b_{1}, b_{2}\in D'$ and $a_{2}\not\in D'$, then $D'-b_{2}=C(A'',b_{2})$. Thus
\begin{equation}\label{eq8}
b_{1}\in C(A'',b_{2}).
\end{equation}
Recall that in (\ref{eq1}) we had that $b_{2}\in C(B',b_{1})$ and in (\ref{eq3}) we had that $b_{2}\not\in C(B',a_{2})$. It follows from Lemma~\ref{lem3} that after exchanging $a_{2}$ and $a_{1}$ relative to $A'$ and $B'$ we must have that
\begin{equation}\label{eq9}
b_{2}\in C(B'',b_{1}).
\end{equation}
From (\ref{eq8}), (\ref{eq9}) we have that $b_{1}$ and $b_{2}$ are symmetrically exchangeable relative to $A''$ and $B''$.

In (\ref{eq2}) we had that $b_{1}\in C(A',b_{3})$ and in (\ref{eq3}) we had that $a_{2}\not\in C(A',b_{3})$. It follows from Lemma~\ref{lem3} that after exchanging $a_{2}$ with $a_{1}$, relative to $A'$ and $B'$, we must have that $b_{1}\in C(A'',b_{3})$. By Lemma~\ref{lem4}, $C(B'',a_1)=C(B',a_2)-a_1+a_2$ and for $b'\in B'-a_1$, $a_1\in C(A'',b')$ if and only if $a_2\in C(A',b')$. Thus $b_3\in C(B'',a_1)$ and $a_1\in C(A'',b_2)$ and we are back in the setup of Case 1 with $A'$ and $B'$ replaced by $A''$ and $B''$ respectively, and exchanging the roles of $a_{1}$ and $a_{2}$. Following the arguments of Case 1 we obtain the bases $A^{**}=A''-b_{1}+b_{2}$ and $B^{**}=B''-b_{2}+b_{1}$ and the symmetric exchange $(a_{1}, b_{3})$ relative to $A^{**}$ and $B^{**}$. Since $A^{**}$ and $B^{**}$ can be obtained from $A$ and $B$ by the single symmetric exchange $(a_{2}, b_{2})$, we conclude that the sets $\{a_{1},a_{2}\}$ and $\{b_{2},b_{3}\}$ constitute a serial symmetric exchange relative to $A$ and $B$. This completes the proof of Theorem~\ref{thm1}.
\end{proof}

\section{The Case $\rho(M)=4$}
In this section $M$ is a matroid of rank 4 consisting of two bases $A=\{a_{1},a_{2},a_{3},a_{4}\}$ and $B=\{b_{1},b_{2},b_{3},b_{4}\}$.
\begin{prop}\label{prop1}
Let $\{a_{1},a_{2}\}$ and $\{b_{1},b_{2}\}$ be a serial symmetric exchange relative to $A$ and $B$. Let $a'\in A\setminus \{a_{1},a_{2}\}$ and $b'\in B\setminus \{b_{1},b_{2}\}$. If $(a',b')$ is a symmetric exchange relative to $A$ and $B$, then the pair $(A,B)$ is a serial symmetric exchange.
\end{prop}
\begin{proof}
Let $A'=\{b_{1},b_{2},a_{3},a_{4}\}$ and $B'=\{a_{1},a_{2},b_{3},b_{4}\}$ be the bases obtained after serially exchanging $\{a_{1},a_{2}\}$ with $\{b_{1},b_{2}\}$. Suppose that $(a_{4},b_{4})$ is a symmetric exchange relative to $A$ and $B$, so that $\{a_{1},a_{2},a_{3},b_{4}\}$ and $\{b_{1},b_{2},b_{3},a_{4}\}$ are bases. These two bases can also be obtained by performing the exchange $(a_{3},b_{3})$ on the bases $A'$ and $B'$.
\end{proof}

\begin{thm}
Conjecture~\ref{11::22} holds when $\rho(M)=4$.
\end{thm}
\begin{proof}
From Theorem~\ref{thm1} we may assume, without loss of generality, that the sets $\{a_{1},a_{2}\}$ and $\{b_{1},b_{2}\}$ constitute a serial symmetric exchange relative to $A$ and $B$. If among the remaining elements there is a symmetric exchange relative to $A$ and $B$ we are done, by Proposition~\ref{prop1}. So, we assume that there is no symmetric exchange relative to $A$ and $B$ among $\{a_{3},a_{4}\}$ and $\{b_{3},b_{4}\}$. By Theorem~\ref{thm1}, the pair $\{a_{3},a_{4}\}$ and some pair of $B$-elements form a serial symmetric exchange. This pair must exclude at least one of $b_{3}$ and $b_{4}$ (since there is no symmetric exchange relative to $A$ and $B$ between $\{a_{3},a_{4}\}$ and $\{b_{3},b_{4}\}$, the first exchange must involve either $b_{1}$ or $b_{2}$). After serially exchanging $\{a_{3},a_{4}\}$ with a pair of elements of $B$ we are left with $\{a_{1},a_{2}\}$ on the $A$ side, and at least one of $b_{3}$ and $b_{4}$ on the $B$ side. These remaining elements must contain a symmetric exchange relative to $A$ and $B$, since both $b_{3}$ and $b_{4}$ have symmetric exchanges with $a_{1}$ or $a_{2}$ relative to $B$ and $A$ (they have no symmetric exchange with either $a_{3}$ or $a_{4}$ relative to $B$ and $A$). Hence, by Proposition~\ref{prop1}, $A$ and $B$ form a serial symmetric exchange.
\end{proof}

\bibliographystyle{amsplain}
\bibliography{kzrefs}
\end{document}